\documentclass[a5paper, 10pt]{article}

\usepackage{cite, amsmath, amssymb, booktabs, lastpage, graphicx}
\usepackage[hidelinks]{hyperref}

\usepackage{geometry}
\usepackage{setspace}
 \usepackage{cite, amsmath, amssymb, amsthm}
    \usepackage[mathscr]{euscript}
    \usepackage{graphicx}
    \usepackage[margin=1cm,%
                font=small,%
                format=hang,%
                labelsep=period,%
                labelfont=bf]{caption}

\usepackage{amsthm}

\theoremstyle{plain}
\newtheorem{lemma}{Lemma}
\newtheorem{theorem}{Theorem}

\newtheorem{proposition}[theorem]{Proposition}

\theoremstyle{remark}

\theoremstyle{definition}

\usepackage{graphicx}
\makeatletter
\renewcommand{\maketitle}{
        \begin{center}

                {\Large\bfseries \@title} \par
                \vspace{5mm}
                \baselineskip=0.2in
                {\large\bfseries \@author}\par
                \vspace{1mm}
                {\it \@address} \par
                {\small\tt \@email} \par
                \vspace{3mm}
                
        \end{center}
        \vspace{3mm}
}
\newcommand{\address}[1]{\def\@address{#1}}
\newcommand{\email}[1]{\def\@email{#1}}

\makeatother

\usepackage{fancyhdr}

\usepackage[margin=1cm,%
                        font=footnotesize,%
                        format=hang,%
                        labelsep=period,%
                        labelfont=bf]{caption}

\title{Some New Results on Seidel Equienergetic Graphs}
\author{Samir K. Vaidya$^{a}$, 
        Kalpesh M. Popat$^b$}

\address{$^a$ Department of Mathematics, Saurashtra University, Rajkot, India\\
        $^b$ Department of Mathematics, Saurashtra University, Rajkot, India
    }

\email{samirkvaidya@yahoo.co.in, kalpeshmpopat@gmail.com}

\newcommand{\boxtensor}{{\Box\keCahitrn-9.03pt\raise1.42pt\hbox{$\times$}}}

\newcommand{\be}{\begin{eqnarray}}
\newcommand{\ee}{\end{eqnarray}}

\begin{document}

\maketitle

\begin{abstract}
The energy of a graph $G$ is the sum of the absolute values of the eigenvalues of the adjacency matrix of $G$. Some variants of energy can also be found in the literature which are defined on the concepts of Laplacian matrix, Distance matrix, Common-neighbourhood matrix and Seidel matrix. The Seidel matrix of the graph $G$ is the square matrix in which $ij^{th}$ entry is $-1$ or $1$, if the vertices $v_i$ and $v_j$ are adjacent or non-adjacent respectively, and is $0$ , if $v_i=v_j.$ The Seidel energy of $G$ is the sum of the absolute values of the eigenvalues of its Seidel matrix. We present here some graph families which are Seidel equienergetic.
\end{abstract}

\section{Introduction}
For standard terminology and notations related to graph theory we follow Balakrishnan and Ranganathan \cite{def} while for any undefined term in algebra we follow Lang \cite{lan}. Let $G$ be a simple graph with vertex set $\{v_1,v_2,\cdots,v_n \}$. The \textit{adjacency matrix} denoted by $A(G)$ of $G$ is defined to be $A(G)=[a_{ij}]$, such that, $a_{ij}=1$ if $v_i$ is adjacent with $v_j$, and $0$ otherwise. 
\par The eigenvalues of $A$ are called the eigenvalues of $G$.  The energy $E(G)$ of graph $G$ is the sum of all absolute values of eigenvalues of $G$. The concept of \textit{energy of graph} was introduced by Gutman \cite{gut1} in 1978.  A brief account on energy of graph can be found in Cvetkovi$\Grave{c}$ \cite{cve} and Li \cite{li}
\par The other varients of energy like Laplacian energy \cite{lap}, Incidence energy \cite{gut3}, Skew energy \cite{adi}, Distance energy \cite{boz}, Seidel energy \cite{hea} are also available in the literature. In the present paper we have focused on Seidel energy of graphs. 
\par Let $G$ be a simple graph with vertex set $\{v_1,v_2,\cdots,v_n \}$. The \textit{Seidel matrix} of $G$ which is denoted by $S(G) = [s_{ij}]$ is a $n \times n$ matrix in which $s_{11} = s_{22} = \cdots = s_{nn}=0$. Also for $ i \neq j, s_{ij}= -1$, if $v_i$ is adjacent to $v_j$ and $s_{ij}=1$, otherwise.  
\par The eigenvalues of the Seidel matrix, labeled as $\sigma_1,\sigma_2,\cdots,\sigma_n$, are said to be the Seidel eigenvalues of $G$.  For convenient we represent the multiplicity of any eigenvalue as its power.  The collection of Seidel eigenvalues together with their multiplicities is known as Seidel spectrum of $G$ denoted by $Spec_s(G)$.
Haemers \cite{hea} has defined the Seidel energy of $G$ as $$SE(G)=\sum_{i=1}^{n} | \sigma_i |$$
As an example the, Seidel matrix of the complete graph $K_n$ is $I-J$. Thus $$Spec_s(K_n)=\{1^{n-1}\} \cup \{1-n^{1}\}$$ Therefore, $SE(K_n) = 2n - 2$. 
\par Two graphs $G_1$ and $G_2$ are said to be Seidel equienergetic if $SE(G_1) = SE(G_2)$. Of course, Seidel cospectral graphs are Seidel equienergetic. We are interested in graphs which are of  same order, non co-spectral and equienergetic in the context of Seidel energy. Let $\overline{G}$ be the complement of the graph $G$ then $S(G) = A(G)-A(\overline{G})$ implying $S(\overline{G})=-S(G)$. Consequently, $G$ and $\overline{G}$ are obviously Seidel equienergetic. 
\par The \textit{join of two graphs} $G_1$ and $G_2$, denoted by $G_1 + G_2$, is a graph obtained from $G_1 \cup G_2$ by joining each vertex of $G_1$ to all vertices of $G_2$. Ramane \textit{et al.} \cite{ram1} have proved that if $H_1$ and $H_2$ are Seidel non cospectral, Seidel equienergetic regular graphs on $n$
vertices and of same degree, then for any regular graph $G$, $SE(H_1 + G) = SE(H_2 + G)$.
\par Ramane \textit{et al.} \cite{ram} have also proved that if $G_1$ and $G_2$ be two Seidel non-cospectral $r$- regular graphs of the same order and of the same degree with $r \geq 3$, then for any $k \geq 2$, the iterated line graphs $L^k(G_1)$ and $L^k(G_2)$ are Seidel non-cospectral, Seidel equienergetic graphs.
\par The  \textit{cartesian product} of graphs $G$ and $H$ is a graph, denoted as $G \times H$, whose vertex set is $V(G) \times V(H)$. Two vertices $(u_1,v_1)$ and $(u_2,v_2)$ are adjacent if $u_1=u_2$ and $v_1v_2\in E(H)$ or $v_1=v_2$ and $u_1u_2 \in E(G)$. 
\par Let $A \in R^{m\times n}$, $B \in R^{p \times q}$. The \textit{Kronecker product} (or tensor product) of $A$ and $B$ is defined as the matrix 
\[A \otimes B = \begin{bmatrix}
a_{11}B & \cdots & a_{1n}B\\ 
\vdots & \ddots & \vdots \\
a_{m1}B & \cdots & a_{mn}B
\end{bmatrix} \]
\par It is well known and easy to show that for graphs $G_1$ and $G_2$ for which loops are allowed, we have
$$A(G_1 \times G_2)=A(G_1) \otimes A(G_2)$$
\par We define $G^r$ be the graph obtained from $G$ by adding loops on every vertex and $G^{ur}$ be the graph obtained from $G$ by removing all loops. For any simple graph $G$, if we take $G_1=K_m^r$ and $G_2=G$, using the facts that $A(G^r)=A(G)+I, A(G^{ur})=A(G)-I, A(K_m^r)=J$ and writing $A$ for $A(G)$, we have, 
$$A(K_m^r \times G)=J \otimes A$$
and
$$A(K_m^r \times G^r)^{ur}=J \otimes (A+I)-I \otimes I$$
\hspace{0.5cm} We introduce two convenient notations,  $D_m(G)=K_m^r \times G$ and $D_m^*(G)=(K_m^r \times G^r)^{ur}$. It is easy to verify that if $G$ is a graph with $n$ vertices then both  $D_m(G)$ and  $D_m^*(G)$ are graphs with $mn$ vertices. For better understanding  $D_2(C_5)$ and $D_2^*(C_5)$ are shown in following figure.
\begin{figure}[h]
\includegraphics[scale=0.4]{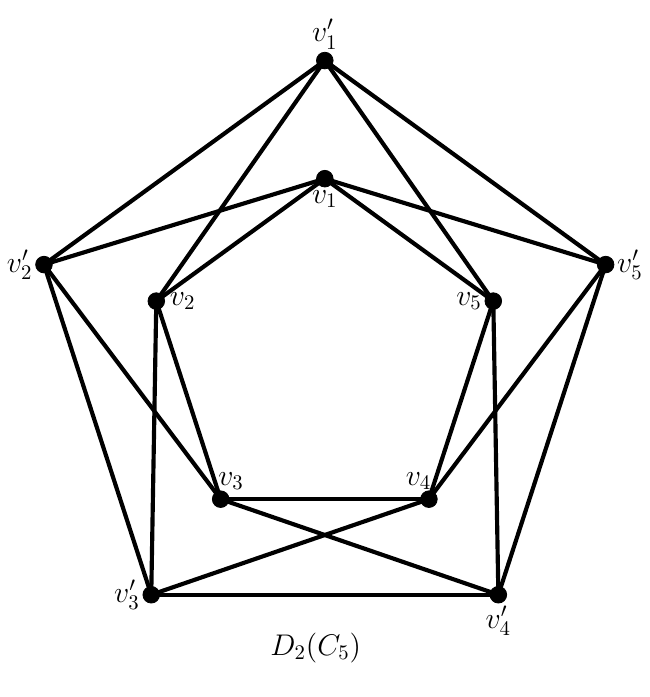}
\hfill
\includegraphics[scale=0.4]{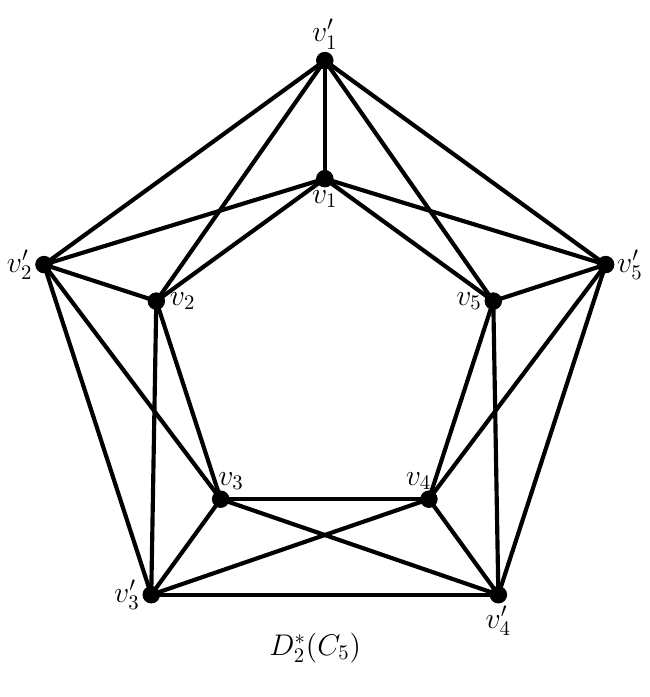}
\caption{}
\end{figure} 

\begin{proposition}\label{prop1} \cite{hor}
Let $A \in M^m$ and $B \in M^n$. Furthermore, let $\sigma$ be eigenvalue of matrix $A$ with corresponding eigenvector $x$ and $\mu$ be eigenvalue of matrix $B$ with corresponding eigenvector $y$. Then $\sigma \mu$ is an eigenvalue of $A \otimes B$ with corresponding eigenvector $x \otimes y$.
\end{proposition}
\section{Seidel Equienergetic graphs}
\begin{lemma} \label{lemma1}
If the Seidel spectrum of $G$ is $\{\sigma_1,\sigma_2,\cdots,\sigma_n\}$ then the Seidel spectrum of $D_m(G)$ is 
$$Spec_s(D_m(G))=\{m\sigma_i+(m-1)|\,i=1,2,\cdots,n\}\cup \{-1^{mn-n}\}$$
\end{lemma}
\begin{proof}
As $A(K_m^r)=J=J_m$ has spectrum $\{m,0^{m-1}\}$, we have by Proposition \ref{prop1} that an eigenvalue $\sigma_i$ of $S(G)$ yields in $A(K_m^r)\times (S(G)+I)-I$ an eigenvalue $m(\sigma_i+1)-1=m\sigma_i+(m-1)$ and $mn-n$ eigenvalue $0(\sigma_i+1)-1=-1$. It is enough to show that $S(K_m^r \times G)=A(K_m^r) \times (S(G)+I)-I$. Writing $A$ for $A(G)$ and $S(M)=-2M-I+J$ for any matrix $M$, We have to show that $S(J \otimes A)=J \otimes (S(A)+I)-I$. Observe that first,
$$S(B \otimes A)=-2(B \otimes A)-I \otimes I+ J \otimes J$$
and so taking $B=J$ and moving $I \times I$ to the other side, $S(K_m^r \times G)+I$ is 
$$S(J \otimes A)+I \otimes I=-2(J \otimes A)+J \otimes J=J \otimes(-2A+J)=J \otimes (S(A)+I)$$
\end{proof}
\begin{lemma} \label{lemma2}
If the Seidel spectrum of $G$ is $\{\sigma_1,\sigma_2,\cdots,\sigma_n\}$ then the Seidel spectrum of $D_m^*(G)$ is 
$$Spec_s(D_m^*(G))=\{m\sigma_i-(m-1)|i=1,2,\cdots,n\} \cup \{1^{mn-m}\}$$
\end{lemma}
\begin{proof}
Much like above,
\begin{align*}
\left[S(J \otimes (A+I)-I\otimes I\right]+I= & \left[-2(J \otimes(A+I)-I\otimes I+J \otimes J-I \otimes I \right]+I \otimes I \\
= & J \otimes (-2(A+I))+ J \otimes J \\
= & J \otimes (-2(A+I)+J) = J \otimes (S(A)-I)
 \end{align*}
\end{proof}
\begin{theorem} \label{thm1}
Let $G$ be a graph with Seidel eigenvalues $\sigma_1, \sigma_2,\cdots, \sigma_n$ with $\left| \sigma_i \right| \geq \frac{m-1}{m}$ , for all $1 \leq i \leq n $ then, $D_m(G)$ and $D_m^*(G)$ are Seidel non co-spectral equienergetic graphs if and only if $G$ have equal numbers of positive and negative Seidel eigenvalues.
\end{theorem}
\begin{proof}
It is not hard to see that the sums of the spectrums in these two lemma are the same. To say the same about the sums of the absolute values, (so that the Seidel energies are the same) we observe that we need that
$$\sum_{i}\left|m\sigma_i+(m-1)\right|=\sum_{i} \left|m\sigma_i-(m-1)\right|$$
Assuming that $|\sigma_i|>\dfrac{m-1}{m}$, we get that $\left|m\sigma_i+(m-1)\right|-\left|m\sigma_i-(m-1)\right|$ is $2(m-1)$ if $\sigma_i$ is positive and $-2(m-1)$ if $\sigma_i$ is negative for each $i=1,2,\cdots,n$. It follows that with these restrictions on the eigenvalues that the Seidel energies of $D_m(G)$ and $D_m^*(G)$ are the same if and only if $G$ has the same number of positive and negative eigenvalues. This gives Theorem \ref{thm1}.
\end{proof}
\par Observe that if $G$ is a graph with $n$ vertices then both  $D_m(D_m^*(G))$ and  $D_m^*(D_m(G))$ are graphs with $m^2n$ vertices. We next prove following theorem for Seidel equienergetic graphs.
\begin{theorem}
Let $G$ be a graph with Seidel eigenvalues $\sigma_1, \sigma_2,\cdots, \sigma_n$ with $\left| \sigma_i \right| \geq \left(\frac{m-1}{m}\right)^2$ , for all $1 \leq i \leq n $ then, $D_m^*(D_m(G))$ and $D_m(D_m^*(G))$ are Seidel non co-spectral equienergetic graphs if and only if $G$ has equal numbers of positive and negative Seidel eigenvalues.
\end{theorem}
\begin{proof}
By Lemma \ref{lemma1}, $D_m(G)$ has spectrum $\{m\sigma_i+(m-1)|\,i=1,2,\cdots,n\} \cup \{-1^{mn-n}\}$ and by Lemma \ref{lemma2} spectrum of $D_m^*(D_m(G))$ is $\{m^2\sigma_i+(m-1)^2|\, i=1,2,\cdots,n\} \cup \{(1-2m)^{mn-n}\} \cup \{1^{m^2n-mn}\}$ and again by Lemma \ref{lemma2} and Lemma \ref{lemma1}, spectrum of $D_m(D_m^*(G))$ is $\{m^2\sigma_i-(m-1)^2|\, i=1,2,\cdots,n\} \cup \{(2m-1)^{mn-n}\} \cup \{-1^{m^2n-mn}\}$. To prove $D_m^*(D_m(G))$ and $D_m(D_m^*(G))$ are equienergetic, it is enough to prove 
$$\sum_{i}\left|m^2\sigma_i+(m-1)^2\right|=\sum_{i} \left|m^2\sigma_i-(m-1)^2\right|.$$  If $\left|\sigma_i\right|>\dfrac{(m-1)^2}{m^2}$, we have $\left|m^2\sigma_i+(m-1)^2\right|-\left|m^2\sigma_i-(m-1)^2\right|$ is $2(m-1)^2$ if $\sigma_i$ is positive and $-2(m-1)^2$ if $\sigma_i$ is negative. It follows that $D_m^*(D_m(G))$ and $D_m(D_m^*(G))$ are Seidel equienergetic if and only if $G$ has equal numbers of positive and negative Seidel eigenvalues.
\end{proof}

\section{Concluding Remarks}
The concept of Seidel equienergetic graphs is analogous to the concepts of equienergetic graphs. We present here methods to construct Seidel equienergetic graphs by means of $G^r$ and $G^{ur}$ from a given graph $G$.     

\section*{Acknowledgement}
The authors are highly thankful to the anonymous  referee for kind suggestions and comments on the first draft of this paper.


\end{document}